\documentclass{amsart}

\usepackage{multicol}
\usepackage{caption,hyperref}
\usepackage{enumerate,amsmath,amssymb,graphicx}
\usepackage{subfig}

\usepackage{color}

\def\RR{\mathbb R}
\def\NN{\mathbb N}
\def\ZZ{\mathbb Z}

\date{}
\newcommand{\bbinomial}[2]{\begin{bmatrix}#1\\#2\end{bmatrix}}

\newtheorem{proposition}{Proposition}
\newtheorem{lemma}{Lemma}
\newtheorem{theorem}{Theorem}

\author{ Mirko D'Ovidio \and Anna Chiara Lai \and Paola Loreti}

\begin{document}
\title{Generalized binomials in fractional calculus}

\maketitle
{\small \begin{center} Dipartimento di Scienze di Base e Applicate per l'Ingegneria,\\
Sapienza Università di Roma,\\ 
Via A. Scarpa 16, 00161 Roma, Italy\end{center}}

\begin{abstract}
We consider a class of generalized binomials emerging in fractional calculus. After establishing some general properties, we focus on a particular yet relevant case, for which we provide several ready-for-use combinatorial identities, including an adapted version of the Pascal's rule. We then investigate the associated generating functions, for which we establish a recursive,  combinatorial and integral formulation. From this, we derive an asymptotic version of the Binomial Theorem.  A combinatorial and asymptotic analysis of some finite sums completes the paper.   
\end{abstract}
	\section{Introduction}
The continuous binomial function is a generalization of the classical binomial coefficient defined by
$$\binom{y}{x}:=\frac{\Gamma(y+1)}{\Gamma(x+1)\Gamma(y-x+1)}, \quad x\in \RR,\,y\in\RR\setminus\ZZ_{\leq -1},$$		
	where $\Gamma(\cdot)$ is the Gamma function
$$\Gamma(t+1):=\lim_{n\to\infty} n^t \prod_{k=0}^n\left(1+\frac{t}{k}\right)^{-1} \quad t\in\RR,\,t\not\in\ZZ_{\leq-1}.$$	
for which we also recall the classical, integral formulation
$$\Gamma(t+1)=\int_0^{+\infty} \tau^{t}e^{-\tau}d\tau \qquad t>-1.$$	
For a fixed real $y$ which is a non-positive integer, the natural domain of the binomial function is given by $X_y:=\{x\in \RR\mid x,y-x\not\in \ZZ_{\leq -1}\}$. However $X_y$ can be continuously extended to the real line by setting   	
$$\binom{y}{x}:=0 \quad x,y-x\in \ZZ_{\leq -1}.$$	
	The resulting continuous binomial function is smooth, symmetric about $y/2$ (where it attains its global maximum), vanishing as $x\to\pm \infty$ and displaying a  graph rensembilng a damped sinusoid \cite{binomialthm}. Moreover, the binomial function is a smooth interpolation of the generalized binomial, defined  for  $x=k$ , with $k\in\NN$ by
	$$\binom{y}{k}:=\begin{cases}\frac{y(y-1)\cdots(y-k+1)}{k!}\quad&k\geq 1\\
	1&k=0\\
	0&k\leq -1.\end{cases}$$
	Also note that, by also assuming $y=n$, $n\in\NN$, we recover the classical binomial coefficient:
	$$\binom{n}{k}:=\frac{n!}{k!(n-k)!}.$$
	This can be readily verified by recalling the recursive relation $\Gamma(1+x)=x\Gamma(x)$ and its direct consequence $\Gamma(1+n)=n!$.
	
The binomial function inherits, in a suitable sense,  many of the combinatorial properties of the binomial coefficient. To give an example, we recall that the Binomial Theorem states the identity 
\begin{equation}\label{binomthm}
\sum_{k=0}^n \binom{n}{k}=2^n\qquad n\in\NN
\end{equation}
and, more generally,  replacing $n$ with some real $y>-1$, the identity (called Generalized Binomial Theorem) holds
$$\sum_{k=0}^{+\infty} \binom{y}{k}=2^{y}\qquad y\in\RR,\, y>-1.$$
Allowing also $k$ to be a real number, it is proved in \cite{binomialthm} the integral formulation of the Binomial Theorem
\begin{equation}\label{integralsum}
\int_{-\infty}^{+\infty} \binom{y}{x}dx=2^{y}, \quad y\in \RR,\,y\not\in \ZZ_{\leq-1}.\end{equation}

The aim of the present paper is to investigate the properties of a particular class of binomial functions, defined for $\beta\in \RR,\beta>0$, $n,k\in\NN$ by
\begin{equation}\label{bbdef}\bbinomial{n}{k}_\beta:=\begin{cases}\dfrac{\Gamma(\beta n+1)}{\Gamma(\beta k+1)\Gamma(\beta(n-k)+1)}\qquad &\text{if }\beta(n-k)\not\in \ZZ_{\leq-1}\\
0 &\text{otherwise.}
\end{cases}\end{equation}
see Figure \ref{gammaplot} for a plot of this generalized binomial coefficient for $\beta=1/2$ and see Figure \ref{gammabeta} for a plot of the global maxima of \eqref{bbdef} with  $\beta$ ranging in $(0,2]$.
\begin{figure}
\subfloat[]{\includegraphics[scale=0.5]{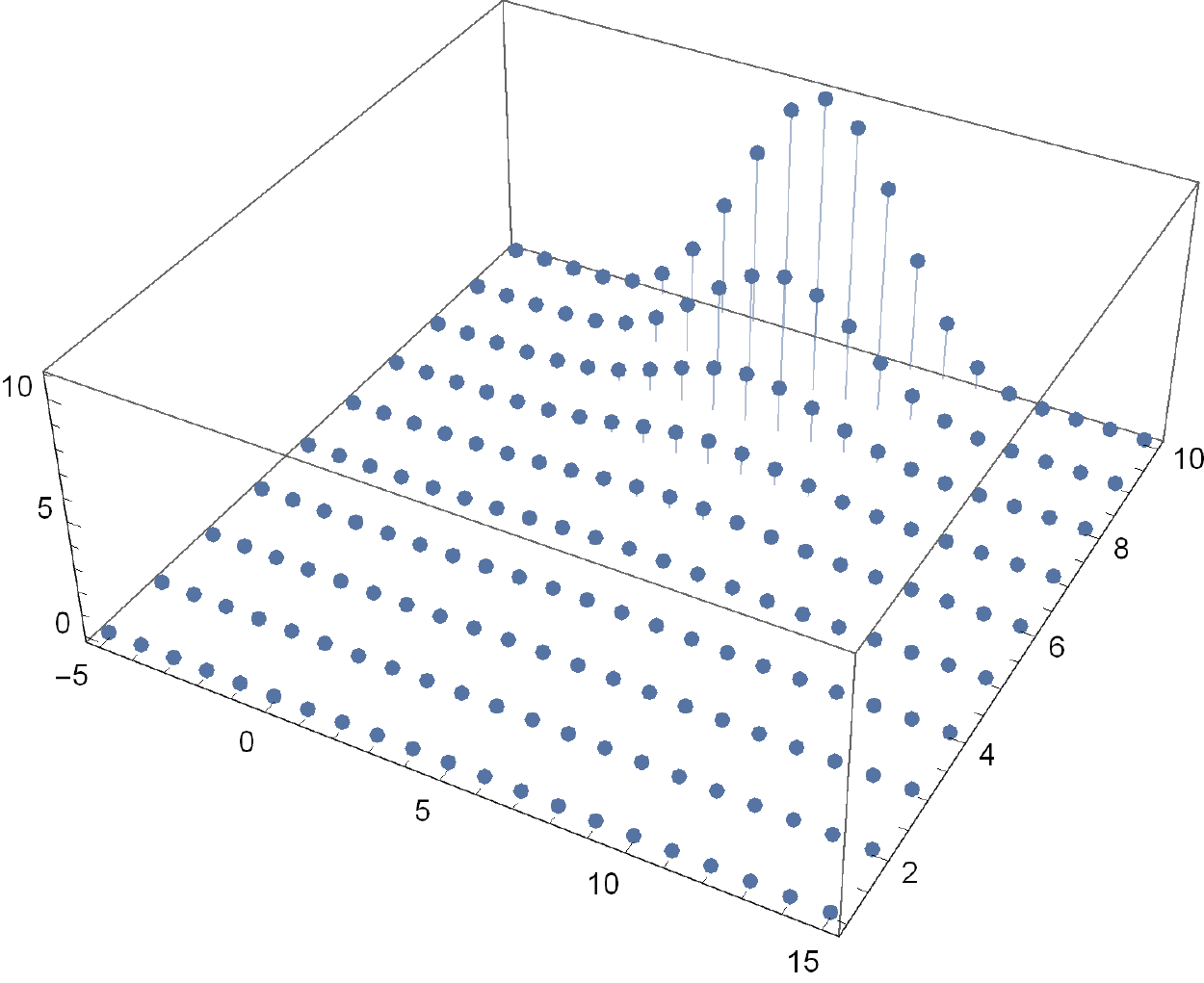}}\subfloat[]{\includegraphics[scale=0.5]{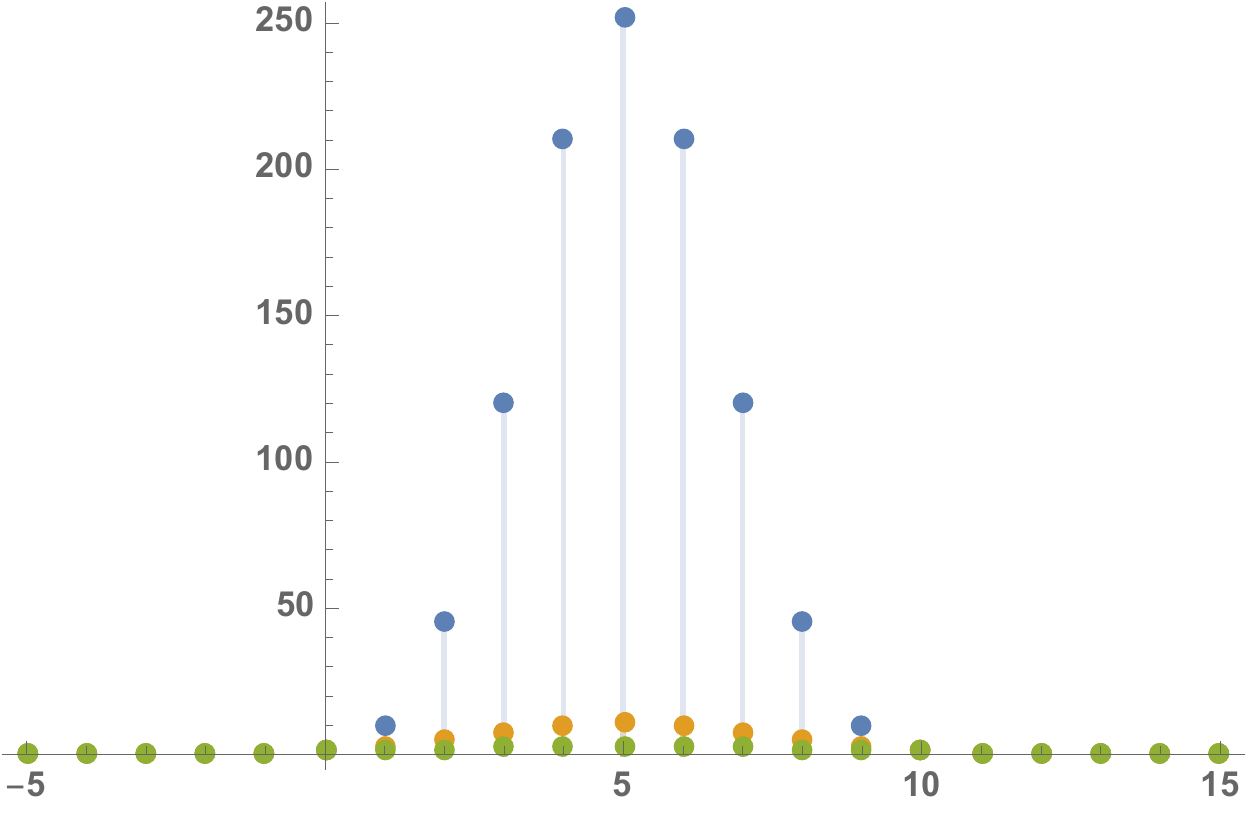}}
\caption{Fractional binomial coefficients for $\beta=1/2$. On the left, the values of $\left[n \atop{k}\right]_{1/2}$ for $n=0,\dots,10$ and $k=-5,\dots,15$ and on the right $\left[ 10 \atop k\right]_{1/2}$ with $k$ ranging between $-5$ and $15$ and $\beta=1/4$ (in blue), $\beta=1/2$ (in orange) and $\beta=1$ (in green).\label{gammaplot}}
\end{figure}
\begin{figure}
\includegraphics[scale=0.4]{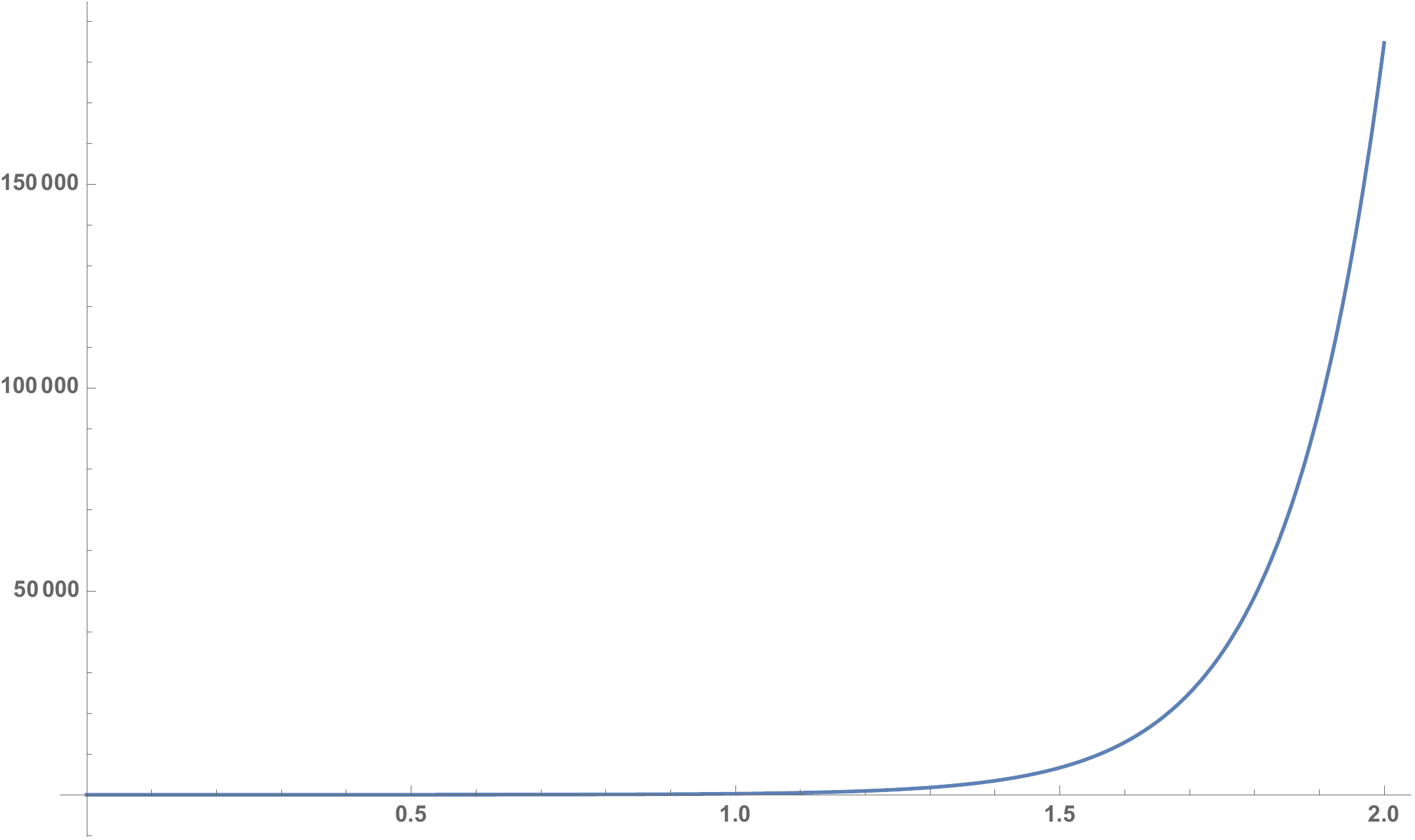}
\caption{For $n=8$, the maximum of the sequence $\left[8 \atop k\right ]_\beta$ (with $k\in\ZZ$) is attained at $k=n/2=4$ for every $\beta>0$. Here we plot the function
$\left[8\atop 4\right]_\beta$ with $\beta\in(0,2]$.\label{gammabeta}}
\end{figure}

The interest in this class of generalized binomial coefficients defined in \eqref{bbdef} is mainly motivated by applications to fractional calculus --  we refer to  \cite{fractional} for an introduction on the topic and as a general reference for the arguments below.{ We only recall here the well-known Caputo (also termed Dzerbayshan-Caputo) derivative of order $\beta \in (0,1)$ of $u : [0, \infty) \to \mathbb{R}$ given by
\begin{align*}
D^\beta u(x) = \frac{1}{\Gamma(1-\beta)} \int_0^x \frac{u^\prime(s)}{(t-s)^\beta} ds 
\end{align*}
where $u^\prime = du/ds$. There are many properties to be noticed for this fractional operator. Consider for instance the Mittag-Leffler functions
$$E_\beta(x):= \frac{1}{2\pi i}\int_{Ha} \frac{\zeta^{\beta -1} e^\zeta}{\zeta^\beta - x}d\zeta =\sum_{n=0}^\infty \frac{1}{\Gamma(\beta n+1)}x^n, \quad \beta >0, \; x \in \mathbb{C}$$  
($Ha$ is the Hankel path) which are widely investigated generalizations of the exponential function in the framework of fractional calculus. 
The Mittag-Leffler functions have been introduced by the Swedish mathematician Mittag-Leffler at the beginning of the last century. 
These completely monotone functions  attracted an increasing attention of mathematicians and applied scientists because of their key role in treating problems related to integral and differential equations of fractional order.

It is well-known that the following relation holds true
\begin{align*}
D^\beta_x \, E_\beta ( \mu \, x^\beta) = \mu \, E_\beta (\mu \, x^\beta), \quad \mu \in \mathbb{R}
\end{align*}
that is, the Mittag-Leffler is an eigenfunction for the Caputo derivative. It solves linear fractional relaxation equations. For $\beta=1$ the Mittag-Leffler becomes the exponential $E_1(x)=e^x$ whereas, for $\beta \in (0,1)$
  we have the following asymptotic behaviours
\begin{align*}
\frac{E_\beta(x)}{e_0(x)} \to 1, \quad \textrm{as} \quad x\to 0 \quad \textrm{and} \quad  \frac{E_\beta(x)}{e_\infty(x)} \to 1, \quad \textrm{as} \quad  x\to \infty
\end{align*} 
where
\begin{align*}
e_0(x) = \exp\left( - \frac{x^\beta}{\Gamma(1+\beta)} \right), \quad \textrm{and} \quad e_\infty(x) = \frac{x^{-\beta}}{\Gamma(1-\beta)}.
\end{align*}

The generalized coefficients considered here emerge naturally when considering the square of $E_{\beta}(x)$.} Indeed, computing the Cauchy product of $E_{\beta}(x)$ for itself, we have
$$ (E_{\beta}(x))^2=\sum_{n=0}^\infty \sum_{k=0}^n \frac{1}{\Gamma(\beta k+1)\Gamma(\beta(n-k)+1)} x^n= \sum_{n=0}^\infty \sum_{k=0}^n \bbinomial{n}{k}_\beta \frac{x^n}{\Gamma(\beta n+1)}.$$
When $\beta=1$ then \eqref{binomthm} readily implies $(E_1(x))^2=E_1(2x)$, namely the elementary property of the exponential function $(e^x)^2=e^{2x}$. Of course, this is not true when replacing $1$ with a generic $\beta>0$, see Figure \ref{MLfig}. {We underline the fact that the Caputo derivative is a non-local operator and therefore it is associated to some memory effect. The Mittag-Leffler turns out to be substantially different from the exponential, in particular
\begin{align*}
E_\beta(x) \, E_\beta(y) \neq E_\beta (x+y).
\end{align*}
Because of this, many problems become very difficult to be treated. This is crucial in the investigation of fractional logistic equation for instance (see \cite{logistic}). An interesting discussion on this point has been recently faced in \cite{DAT}}. The analysis of the Mittag-Leffler function may then benefit of a combinatorial investigation of the partial sum  $\sum_{k=0}^n \left[n\atop{k}\right]_\beta$. 
\begin{figure}	\centering \includegraphics[scale=0.5]{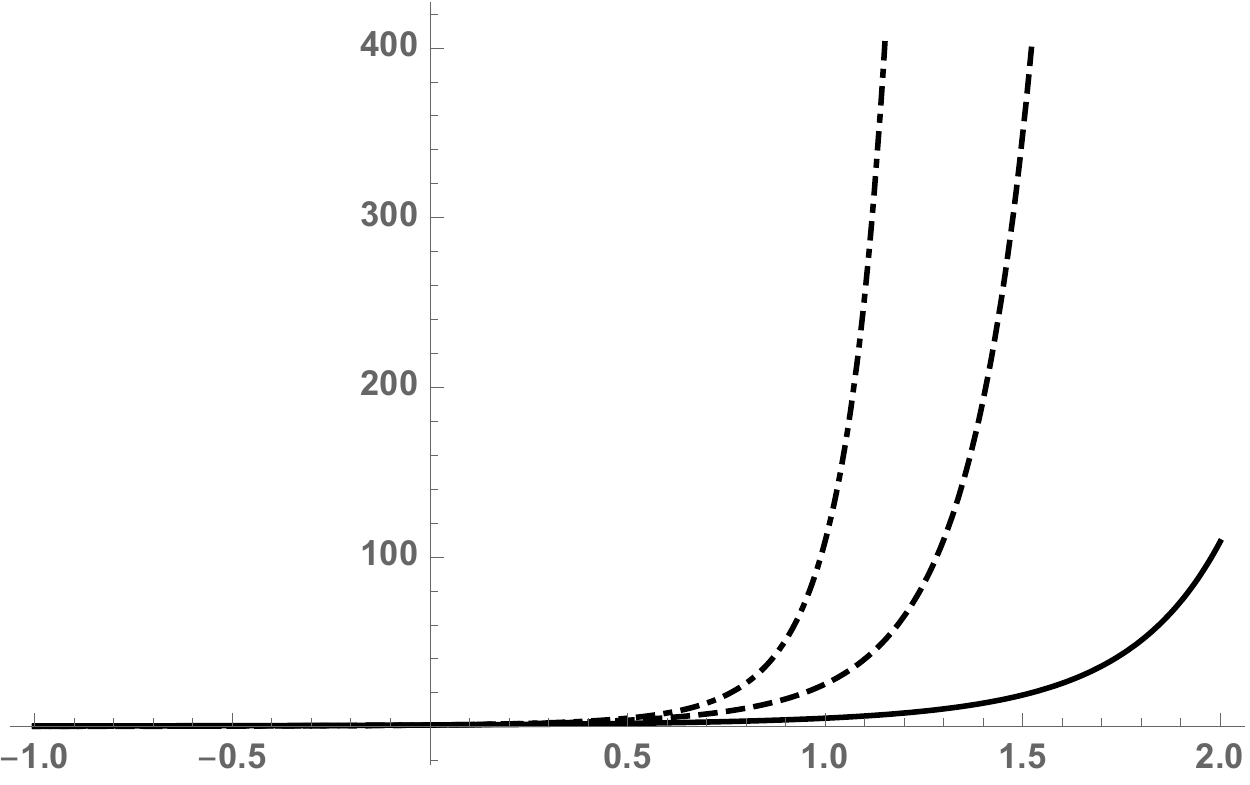}
\caption{The Mittag-Leffler function $E_{1/2}(x)$ (continuous line), its square $(E_{1/2}(x))^2$ (dashed line) and the function $E_{1/2}(2x)$ (dot-dashed line), for $x\in[-0.5,2]$.\label{MLfig}}
\end{figure}

\vskip0.5cm
We also remark that the coefficients \eqref{bbdef} can be expressed in terms of hyper-geometric functions. In particular, we have by Gauss theorem
$$\bbinomial{n}{k}_\beta=\,_2F_1(-\beta k,-\beta (n-k);1;1).$$
where 
$$_2F_1(a,b;c;x):=\frac{\Gamma(c)}{\Gamma(c-b)\Gamma(b)}\int_0^1t^{b-1}(1-t)^{c-b-1}(1-xt)^{-a}dt\qquad  |x|\leq 1.$$

In particular we have for $\beta=1/2$ 
$$\bbinomial{n}{k}_{1/2}=\,_2F_1(-k/2,-(n-k)/2;1;1).$$
Now,  in the case $n=2$ and $k=1$ the function $_2F_1(-1/2,-1/2;1;x^2)$ has a precise geometric interpretation: the quantity 
$ 4\pi (a+b)_2F_1(-1/2,-1/2;1;(a-b)^2/(a+b)^2)$ is indeed the length of the ellipse with semi-axes $a,b>0$. The investigation of the case $\beta=1/2$ (with general $n$ and $k$) may also allow for analogous geometrical interpretations.

\vskip0.5cm
The paper is organized as follows. We begin by stating some basic properties for generic, positive $\beta$ and, in the rest of the paper, we focus on the particular case $\beta=1/2$.
In Section \ref{S2} we collect a set of ready-for-use combinatorial identities  derived by an elementary approach and we stress analogies and differences with the classical case $\beta=1$. In Section \ref{S3} we investigate the generating functions associated to \eqref{bbdef}. Finally, Section \ref{S4} is devoted to the combinatorial and asymptotic analysis of partial sums involving \eqref{bbdef}.

\section{Combinatorial identities for generalized binomials}\label{S2}
We begin our investigation by stating some first properties of the generalized binomials defined in \eqref{bbdef} for a general $\beta>0$.
To this end we preliminarly recall the following identities
\begin{equation}\label{gamma}
\Gamma(0)=\Gamma(1)=1;\quad \Gamma(x+1)= x\Gamma(x)\quad\forall x\in \RR^+.
\end{equation}
In particular, the latter equation implies $\Gamma(n+1)=n!$ for all $n\in\NN$. 
\begin{proposition}
For all $\beta>0$ and for all $n\in\NN$
\begin{equation}\bbinomial{n}{0}_\beta=\bbinomial{n}{n}_\beta=1 \label{21gen}.\end{equation}
Moreover the following identities hold
\begin{itemize}
\item[(i)] Reflection formula:
\begin{equation}\label{reflection}
\bbinomial{n}{k}_\beta=\bbinomial{n}{n-k}_\beta \quad \forall k\in \ZZ.
\end{equation}
\item[(ii)]Cancellation identity:
\begin{equation}\label{diff}
\bbinomial{n}{h}_\beta\bbinomial{n-h}{k}_\beta=\bbinomial{n}{k}_\beta\bbinomial{n-k}{h}_\beta\qquad\forall h,k\in\ZZ,\,h,k\leq n. \end{equation}

\end{itemize}
\end{proposition}
\begin{proof}By definition  
$$\bbinomial{n}{k}_\beta=\frac{\Gamma(\beta n+1)}{\Gamma(\beta k+1)\Gamma(\beta(n-k)+1)}=\bbinomial{n}{n-k}_\beta$$
and this proves \eqref{reflection}. Recalling $\Gamma(0)=1$, this also readily yields \eqref{21}. Finally, the Cancellation identity follows by 
\begin{align*}
\bbinomial{n}{h}_\beta\bbinomial{n-h}{k}_\beta&=\frac{\Gamma(\beta n+1)}{\Gamma(\beta h+1)\Gamma(\beta(n-h)+1)}\cdot \frac{\Gamma(\beta(n-h)+1)}{\Gamma(\beta k+1)\Gamma(\beta(n-h-k)+1)}\\
&=\frac{\Gamma(\beta n+1)}{\Gamma(\beta k+1)\Gamma(\beta(n-k)+1)}\cdot \frac{\Gamma(\beta(n-k)+1)}{\Gamma(\beta h+1)\Gamma(\beta(n-h-k)+1)}\\
&=\bbinomial{n}{k}_\beta\bbinomial{n-k}{h}_\beta.
\end{align*}
\end{proof}
\subsection{Combinatorial identities for $\beta=1/2$}

In order to lighten the notation, we set
$$ \bbinomial{n}{k}:=\bbinomial{n}{k}_{1/2}=\frac{\Gamma(\frac{1}{2} n+1)}{\Gamma(\frac{1}{2} k+1)\Gamma(\frac{1}{2}(n-k)+1)}.$$
We make use in what follows of the \emph{duplication formula}
\begin{equation}\label{gammadup}
\Gamma(x)\Gamma(x+1/2)=2^{1-2x}\sqrt{\pi}\Gamma(2x)\quad \forall x\in \RR^+.
\end{equation}
 which also reads, for $x=n/2$,  
$$\Gamma\left(\frac{1}{2}(n+1)\right)=2^{1-n}\sqrt{\pi}\frac{\Gamma(n)}{\Gamma(n/2)}.$$

We begin our investigation with some closed form expression. 
\begin{proposition}
For all $n\in\NN$ and for all $k\in\ZZ$:
\begin{align}
&\bbinomial{n}{0}=\bbinomial{n}{n}=1 \label{21}\\
&\bbinomial{n}{1}=
\bbinomial{n}{n-1}=\frac{2 \Gamma(\frac{n}{2}+1)}{\sqrt{\pi}\Gamma(\frac{n}{2}+\frac{1}{2})}\label{22}\\
&\bbinomial{n}{-1}=
\bbinomial{n}{n+1}=\frac{1}{n+1}\bbinomial{n}{1} \label{23}\\
&\bbinomial{n}{-2k}=
\bbinomial{n}{n+2k}=0.\label{24}
\end{align}
A more explicit formula for \eqref{22} is
\begin{align}
&\bbinomial{2n}{1}=\frac{2^{2n+1}}{\pi\binom{2n}{n}}\label{21even}\\
&\bbinomial{2n+1}{1}=\frac{2n+1}{2^{2n}}\binom{2n}{n}\label{21odd}.
\end{align}
Finally we have
\begin{equation}\label{eveneven}\bbinomial{2n}{2k}=\binom{n}{k}.\end{equation}
\end{proposition}
\begin{proof}
Equations \eqref{21} is reported for completeness, indeed it was already stated for general $\beta$ in \eqref{21gen}. The equations \eqref{22} and \eqref{eveneven} readily follow from the definition and from the fact that $\Gamma(3/2)=\sqrt{\pi}/2$. To prove \eqref{23} it suffices to remark that
$$\bbinomial{n}{-1}=\frac{\Gamma\left(\frac{n}{2}+1\right)}{\Gamma\left(\frac{1}{2}\right)\Gamma\left(\frac{n+1}{2}+1\right)}=\frac{1}{n+1}\frac{\Gamma\left(\frac{n}{2}+1\right)}{\Gamma\left(\frac{3}{2}\right)\Gamma\left(\frac{n-1}{2}+1\right)}=\frac{1}{n+1}\bbinomial{n}{1}$$
and to apply the reflection formula to discuss the case $\left[n\atop{n+1}\right]$. The identity \eqref{24} follows by the definition  \eqref{bbdef}.
Finally, equations \eqref{21even} and \eqref{21odd} follow from 
\eqref{22} and by the duplication formula \eqref{gammadup}.\end{proof}

Now, we present some  recursive relations.
\begin{proposition}
For all $n\in\NN$ and for all $k\in \ZZ_{\geq -1}$, we have the following equalities.
\begin{enumerate}
\item[(i)]{Committee/Chair identity}
\begin{equation}\label{rec1}
\bbinomial{n+2}{k+2}=\frac{n+2}{k+2}\bbinomial{n}{k} .\end{equation}
\item[(ii)]{A recursive formula}
\begin{equation}\label{rec2}
\bbinomial{n}{k+2}=\frac{n-k}{k+2}\bbinomial{n}{k}.\end{equation}
\item[(iii)]{A difference formula}
\begin{equation}\label{diff}
\bbinomial{n}{k+2}-\bbinomial{n}{k}=\frac{n-2-2k}{n+2}\bbinomial{n+2}{k+2}.\end{equation}
\end{enumerate}
\end{proposition}
\begin{proof} The Committee/Chair identity \eqref{rec1}  is a consequence of \eqref{gamma}, indeed one has for $1/2(n-k)\not\in\ZZ_{\leq -1}$
\begin{align*}\bbinomial{n+2}{k+2}&=\frac{\Gamma\left(\frac{n}{2}+2\right)}{\Gamma\left(\frac{k}{2}+2\right)\Gamma\left(\frac{1}{2}(n-k)+1\right)}\\
&=\frac{n/2+1}{k/2+1}\frac{\Gamma\left(\frac{n}{2}+1\right)}{\Gamma\left(\frac{k}{2}+1\right)\Gamma\left(\frac{1}{2}(n-k)+1\right)}=\frac{n+2}{k+2}\bbinomial{n}{k}.
\end{align*}
If otherwise $1/2(n-k)\in\ZZ_{\leq -1}$then $k\geq n+2$ and, by definition \eqref{bbdef}, \eqref{rec1} reduces to the $0=0$ indentity.  
Similarly, one gets \eqref{rec2} from
\begin{align*}
\bbinomial{n}{k+2}=\frac{\Gamma(\frac{n}{2}+1)}{\Gamma(\frac{k}{2}+2)\Gamma(\frac{1}{2}(n-k))}=\frac{(n-k)/2}{k/2+1}\bbinomial{n}{k}.
\end{align*}
if $1/2(n-k)\not\in\ZZ_{\leq -1}$ and from the trivial identity $0=0$ otherwise.
The Difference formula readily follows replacing $\bbinomial{n+2}{k+2}$ in \eqref{diff} with the right handside of \eqref{rec2}.
\end{proof}

\begin{table}
{\small \begin{tabular}{|l|c|c|}\hline
Formula&$\beta=1$&$\beta>0$\\
\hline&&\\
Reflection&$\displaystyle{\binom{n}{k}=\binom{n}{n-k}}$&$\bbinomial{n}{k}=\bbinomial{n}{n-k}$\\&&\\
Cancelation Identity &$\displaystyle{\binom{n}{h}\binom{n-h}{k}=\binom{n}{k}\binom{n-k}{h}}$&$\bbinomial{n}{h}\bbinomial{n-h}{k}=\bbinomial{n}{k}\bbinomial{n-k}{h}$\\
($h,k\leq n$)&&\\
 &&$\bbinomial{n}{1}\bbinomial{n-1}{k}=\bbinomial{n}{k}\bbinomial{n-k}{1}$\\
&&\\
\hline
\end{tabular}}
\caption{\normalsize Exact formulas for $n\in\NN$, $k,h\in\ZZ$, $\beta>0$. }
\end{table}
\begin{table}[h!]
{\small\centering
\begin{tabular}{|l|c|c|}
\hline&$\beta=1$&$\beta=1/2$\\
\hline&&\\
Value for $k=2$&$\displaystyle{\binom{n}{2}=\frac{n(n-1)}{2}}$\qquad&$\bbinomial{n}{2}=\dfrac{n}{2}$
\\&&\\\hline&&\\
Value for $k=1$&$\displaystyle{\binom{n}{1}=n}$\qquad&
$\bbinomial{2n}{1}=\displaystyle{\frac{2^{2n+1}}{\pi\binom{2n}{n}}}$\\&&\\
&&$\bbinomial{2n+1}{1}=\dfrac{2n+1}{2^{2n}}\displaystyle{\binom{2n}{n}}$\\&&\\\hline&&\\
Value for $k=0$&$\displaystyle{\binom{n}{0}=\binom{n}{n}}=1$\qquad&$\bbinomial{n}{0}=\bbinomial{n}{n}=1$\\&&\\\hline&&\\
Values for $k>n$ &$\displaystyle{\binom{n}{-1}=\binom{n}{n+1}}:=0$\qquad&$\bbinomial{n}{-2}=\bbinomial{n}{n+2}:=0$\\&&\\&&$\bbinomial{n}{-1}=\bbinomial{n}{n+1}=\dfrac{1}{n+1}\bbinomial{n}{1}$\\&&\\
\hline&&\\
{Committee/Chair}&$\displaystyle{\binom{n+1}{k+1}=\frac{n+1}{k+1}\binom{n}{k}}$&$\bbinomial{n+2}{k+2}=\dfrac{n+2}{k+2}\bbinomial{n}{k}$\\ identity&&\\&&\\

Recursion  &$\displaystyle{\binom{n}{k+1}=\frac{n-k}{k+1}\binom{n}{k}}$&$\bbinomial{n}{k+2}=\dfrac{n-k}{k+2}\bbinomial{n}{k}.$\\&&\\
Difference &$\displaystyle{\binom{n}{k+1}-\binom{n}{k}=\frac{n-k-1}{n+1}\binom{n+1}{k+1}}$&$\bbinomial{n}{k+2}-\bbinomial{n}{k}=\dfrac{n-2k-2}{n}\bbinomial{n+2}{k+2}$\\&&\\\hline&&\\
A  sum &$\displaystyle{\binom{n}{2}+\binom{n-1}{2}=(n-1)^2}$&
$\bbinomial{n}{2}+\bbinomial{n-2}{2}=(n-1)=2\bbinomial{n-1}{2}$\\&&\\\hline&&\\

Pascal's rule&$\displaystyle{\binom{n}{k}+\binom{n}{k+1}=\binom{n+1}{k+1}}$&$\bbinomial{n}{k}+
\bbinomial{n}{k+2}
=\bbinomial{n+2}{k+2}$\\&&\\\hline
\end{tabular}}
\caption{\normalsize Exact formulas for $n,k\in\NN$. The identities on the right-most column  also hold for $k=-1$.}
\end{table}

We now establish a recursive relation 
for $1/2$-binomial coefficients generalizing the  celebrated Pascal's rule: 	
$$\binom{n}{k}+\binom{n}{k+1}=\binom{n+1}{k+1}\qquad \forall k,n\in \NN.$$
 \begin{proposition}[Pascal's rule]\label{p1}
For all $n\in \NN$ and for all $k\in \ZZ_{\geq -1}$
\begin{equation}\label{tartagliagen}
\bbinomial{n}{k}+
\bbinomial{n}{k+2}
=\bbinomial{n+2}{k+2}.
\end{equation}
\end{proposition}
\begin{proof}
By \eqref{rec2} and \eqref{rec1} we readily have 
\begin{align*}
\bbinomial{n}{k}+
\bbinomial{n}{k+2}=\frac{n+2}{k+2}\bbinomial{n}{k}=\bbinomial{n+2}{k+2}
\end{align*}
and the proof is complete.
	\end{proof}

%

\section{Generating functions}\label{S3}
We consider the generating function
$$\varphi_n(t):=\sum_{k=0}^\infty \bbinomial{n}{k}t^{k}.$$
Preliminarly, we have the following result
\begin{proposition}
\begin{equation}\label{phi0}\varphi_0(t):=\sum_{k=0}^\infty \bbinomial{0}{k}t^{k}=\frac{2}{\pi}\arctan(t) \quad t\in(-1,1].\end{equation}
\begin{equation}\begin{split}
\varphi_1(t)&:=1+t+\sum_{k=0}^\infty \bbinomial{1}{k}t^{k}=1+t+\frac{t^2}{2}\sum_{k=0}^\infty (-1)^k C_k 2^{-2k} t^{2k}\\
&=t+\sqrt{1+t^2}\hskip3cm t\in[-1,1].\label{phi1}\end{split}\end{equation}
where $C_k:=\frac{1}{k+1}\binom{2k}{k}$ is the $k$-th Catalan number.
\end{proposition}
\begin{proof}
To prove \eqref{phi0}, recall the reflection formula
$$\frac{1}{\Gamma(1+z)\Gamma(1-z)}=\frac{\sin(\pi z)}{z\pi}\qquad z\not\in \ZZ.$$
Applying above identity to $z=k+1/2$, we have by definition \eqref{bbdef}
$$\bbinomial{0}{0}=1;\quad \bbinomial{0}{2k}=0;\quad\bbinomial{0}{1}=\frac{2}{\pi}; \bbinomial{0}{2k+1}=\frac{2}{\pi}\frac{(-1)^k}{2k+1}, \qquad k\in\NN^*$$
from which we deduce
$$\varphi_0(t)=1+\frac{2}{\pi}\sum_{k=0}^\infty \frac{(-1)^k}{2k+1}t^{2k+1}=1+\frac{2}{\pi}\arctan(t)\qquad \text{for } t\in(-1,1].$$
Similarly, we have
$$\bbinomial{1}{2k}=\frac{(-1)^{k+1}}{2k-1}\binom{2k}{k}2^{-2k} \quad k\in\NN;\quad \bbinomial{1}{1}=1;\quad \bbinomial{1}{2k+1}=0, \quad k\in\NN^*$$
and, consenquently, 
\begin{align*}\varphi_1(t)&=1+t+\sum_{k=1}^\infty \frac{(-1)^{k+1}}{2k-1}\binom{2k}{k}2^{-2k}t^{2k}\\&=1+t+\sum_{k=0}^\infty \frac{(-1)^{k}}{2k+1}\binom{2(k+1)}{k+1}2^{-2k-2}t^{2k+2}\\
&=1+t+\frac{t^2}{2}\sum_{k=0}^\infty \frac{(-1)^{k}}{k+1}\binom{2k}{k}2^{-2k}t^{2k}=t+\sqrt{1+t^2}\qquad \text{for } t\in[-1,1].\end{align*}
The latter equality follows by evaluating the generating function of the Catalan numbers $C(x):=\frac{1-\sqrt{1-4x}}{2x}$ (see for instance \cite{catalangen}) in $-t^2/4$.  
\end{proof}
We are now in position to investigate $\varphi_n(t)$ with general $n$. Our starting point is the following recursive relation
\begin{lemma}\label{lemmaphi}
For $n\in\NN$
\begin{equation}\label{recphi}
\varphi_{n+2}(t)=(1+t^2)\varphi_n(t)+\frac{1}{n+1}\bbinomial{n}{1}t.
\end{equation}
\end{lemma}
\begin{proof}
We have by Proposition \ref{p1} (see also Table 1 for the values of $\left[n+2\atop 0\right]$ and $\left[n\atop -1\right]$)
\begin{align*}
\varphi_{n+2}(t):&=\sum_{k=0}^\infty \bbinomial{n+2}{k}t^{k}=1+\bbinomial{n+2}{1}t+\sum_{k=2}^\infty \bbinomial{n+2}{k}t^{k}\\
&=1+\bbinomial{n+2}{1}t+\sum_{k=0}^\infty \bbinomial{n+2}{k+2}t^{k+2}\\
&=1+\bbinomial{n+2}{1}t+\sum_{k=0}^\infty \left(\bbinomial{n}{k}+\bbinomial{n}{k+2}\right)t^{k+2}\\
&=(t^2+1)\varphi_n(t)+\left(\bbinomial{n+2}{1}-\bbinomial{n}{1}\right)t\\
&=(t^2+1)\varphi_n(t)+\frac{1}{n+1}\bbinomial{n}{1}t.
\end{align*}
\end{proof}
Our first main result deals with the asymptotic behaviour of $\varphi_n(t)$ (as $n\to+\infty$) -- see also Figure \ref{figbin}.
\begin{theorem}\label{thm1}For all $n\in\NN$
$$\varphi_{2n}(t)=(1+t^2)^n\left(1+\frac{2}{\pi}\arctan(t)+\frac{t}{\pi}\sum_{k=1}^{n}\frac{1}{k}\frac{2^{2k}}{\binom{2k}{k}}(1+t^2)^{-k}\right);$$
$$\varphi_{2n+1}(t)=(1+t^2)^n\left(\sqrt{1+t^2}+t\sum_{k=0}^{n}\frac{1}{2^{2k}}\binom{2k}{k}(1+t^2)^{-k}\right).$$
Moreover, for all fixed $t\in(-1,1]$
\begin{equation}\label{phiasymeven}\varphi_{2n}(t)\sim(1+t^2)^n\left(1+\frac{4}{\pi}\arctan(t)\right)\qquad \text{as }n\to+\infty;\end{equation}
\begin{equation}\label{phiasymodd}\varphi_{2n+1}(t)\sim(1+t^2)^n\left(1+\text{sign}(t)\right)\sqrt{1+t^2}\qquad \text{as }n\to+\infty.\end{equation}

Finally
\begin{equation}\label{sumbin}
\sum_{k=0}^\infty \bbinomial{n}{k}\sim 2^{n/2+1} \qquad \text{as }n\to+\infty.\end{equation}
\end{theorem}
\begin{proof}
For all $n\in\NN$, using the recursive relation \eqref{recphi} and the identity \eqref{21even} we obtain 
\begin{align*}
\varphi_{2n}(t)&=(1+t^2)^n\left(\varphi_0(t)+t\sum_{k=0}^{n-1}\frac{1}{2k+1}\bbinomial{2k}{1}(1+t^2)^{-k-1}\right)\\
&=(1+t^2)^n\left(\varphi_0(t)+t\sum_{k=0}^{n-1}\frac{1}{2k+1}\frac{2^{2k+1}}{\pi\binom{2k}{k}}(1+t^2)^{-k-1}\right)\\
&=(1+t^2)^n\left(\varphi_0(t)+t\sum_{k=0}^{n-1}\frac{1}{k+1}\frac{2^{2(k+1)}}{\pi\binom{2(k+1)}{k+1}}(1+t^2)^{-k-1}\right)\\
&=(1+t^2)^n\left(\varphi_0(t)+\frac{t}{\pi}\sum_{k=1}^{n}\frac{1}{k}\frac{2^{2k}}{\binom{2k}{k}}(1+t^2)^{-k}\right).
\end{align*}
Now, we recall from \cite{lasalvezza} that the generating function of the sequence $a_n:=\frac{2^{2k}}{k\binom{2k}{k}}$ is convergent in $[-1,1)$ and it reads
 $$W(x):=\sum_{k=1}^\infty \frac{2^{2k}}{k\binom{2k}{k}}x^k=2\sqrt{\frac{x}{1-x}}\arctan{\sqrt{\frac{x}{1-x}}}$$
Therefore for every $t\in(-1,1]$
\begin{align*}\lim_{n\to+\infty}\frac{\varphi_{2n}(t)}{(1+t^2)^n}=&\varphi_0(t)+\frac{t}{\pi}\sum_{k=1}^\infty \frac{1}{k}\frac{2^{2k}}{\binom{2k}{k}}(1+t^2)^{-k}\\=&\varphi_0(t)+\frac{t}{\pi}W((1+t^2)^{-1})=
1+\frac{4}{\pi}\arctan(t).\end{align*}  
Similarly we have
\begin{align*}
\varphi_{2n+1}(t)&=(1+t^2)^n\left(\varphi_1(t)+t\sum_{k=0}^{n-1}\frac{1}{2k+2}\bbinomial{2k+1}{1}(1+t^2)^{-k-1}\right)\\
&=(1+t^2)^n\left(\varphi_1(t)+t\sum_{k=0}^{n-1}\frac{2k+1}{2^{2k}(2k+2)}\binom{2k}{k}(1+t^2)^{-k-1}\right)\\
&=(1+t^2)^n\left(\varphi_1(t)+t\sum_{k=1}^{n}\frac{1}{2^{2k}}\binom{2k}{k}(1+t^2)^{-k}\right).
\end{align*}
Recalling that the generating function of the sequence $b_n:=\frac{1}{2^{2k}}{\binom{2k}{k}}$ is convergent in $(-1,1)$ and it reads
$$Z(x):=\sum_{k=0}^\infty \frac{1}{2^{2k}}{\binom{2k}{k}}x^k=\frac{1}{\sqrt{1-x}}.$$
Therefore for every $t\in(-1,1]$
\begin{align*}
\lim_{n\to+\infty}\frac{\varphi_{2n+1}(t)}{(1+t^2)^n}&=\varphi_1(t)+t\sum_{k=1}^\infty \frac{1}{2^{2k}}{\binom{2k}{k}}(1+t^2)^{-k}\\&=\varphi_1(t)+t (Z((1+t^2)^{-1})-1)=
(1+\text{sign}(t))\sqrt{1+t^2}.\end{align*}

Finally, \eqref{sumbin} follows by choosing $t=1$ in \eqref{phiasymeven} and \eqref{phiasymodd}.

\end{proof}
\begin{figure}
\includegraphics[scale=0.6]{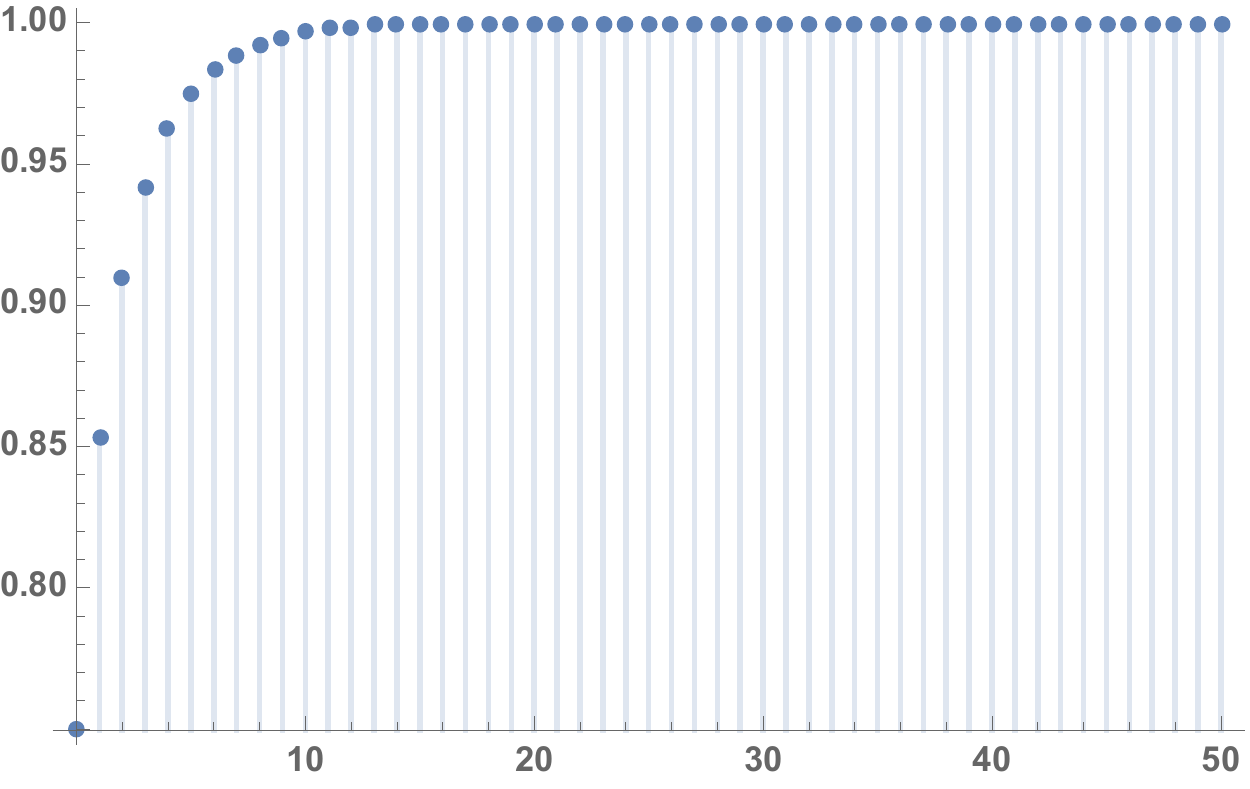}
\caption{The ratio $\varphi_n(1)/2^{n/2+1}$ for $n=0,\dots,50$.\label{figbin}}\end{figure}
We conclude this section with an integral representation for $\varphi_n(t)$.
\begin{theorem}\label{thm2}For all $n\in\NN$
$$\varphi_{n}(t)=(1+t^2)^{n/2}\left(1+\bbinomial{n}{1}\int_0^t\frac{1}{(1+s^2)^{n/2+1}}ds \right).$$
\end{theorem}
\begin{proof}
For a fixed $n\in\NN$ consider the derivative of $\varphi_{n+2}(t)$. By \eqref{rec1} one has for all $t\in\RR$
\begin{equation}\label{diffphi1}\begin{split}\varphi_{n+2}'(t)&:=\sum_{k=1}^{+\infty}k\bbinomial{n+2}{k}t^{k-1}=\bbinomial{n+2}{1}+
(n+2)\sum_{k=2}^{+\infty}\frac{k}{n+2}\bbinomial{n+2}{k}t^{k-1}\\&=\bbinomial{n+2}{1}+
(n+2)\sum_{k=2}^{+\infty}\bbinomial{n}{k-2}t^{k-1}\\&=\bbinomial{n+2}{1}+
(n+2)\sum_{k=0}^{+\infty}\bbinomial{n}{k}t^{k+1}=\bbinomial{n+2}{1}+(n+2)t\varphi_n(t).
\end{split}\end{equation}
On the other hand, Lemma \ref{lemmaphi} implies
\begin{equation}\label{recreldiff}
\begin{split}    \varphi_{n+2}'(t)&=\left((1+t^2)\varphi_n(t)+\frac{1}{n+1}\bbinomial{n}{1}t\right)'\\&=(1+t^2)\varphi'_n(t)+2t\varphi_n(t) +\frac{1}{n+1}\bbinomial{n}{1}
\end{split}\end{equation}
Equating \eqref{diffphi1} and \eqref{recreldiff} obtain  
\begin{align*}\bbinomial{n+2}{1}+(n+2)t\varphi_n(t)=(1+t^2)\varphi'_n(t)+2t\varphi_n(t) +\frac{1}{n+1}\bbinomial{n}{1}.\end{align*}
Therefore $\varphi_n$ is the solution of the linear Cauchy problem
\begin{equation}
    \label{CP}
    \begin{cases}
    (1+t^2)\varphi'-n t\varphi+\bbinomial{n}{1}=0\\
    \varphi(0)=1
    \end{cases}
\end{equation}
i.e.,
$$\varphi_{n}(t)=(1+t^2)^{n/2}\left(1+\bbinomial{n}{1}\int_0^t\frac{1}{(1+s^2)^{n/2+1}}ds \right).$$

\end{proof}
\section{ Some partial sums}\label{S4}
We now investigate the following functions
$$\bar\varphi_n(t):=\sum_{k=0}^n\bbinomial{n}{k}t^k;\qquad n\in\NN.$$
Our first result is an adapted version of Theorems \ref{thm1} and \ref{thm2} to this finitary context. 
\begin{theorem}\label{thmphi}
For all $n\in\NN$,  $\bar\varphi_n(t)$ is the solution of the is the solution of the linear Cauchy problem
\begin{equation}
    \label{CP}
    \begin{cases}
    (1+t^2)\bar\varphi'-n \bar\varphi+\bbinomial{n}{1}(t^n-1)=0\\
    \bar\varphi(0)=1
    \end{cases}
\end{equation}{}
i.e., 
\begin{equation}\label{intform}
\bar\varphi_n(t)=(1+t^2)^{n/2}\left(1-\bbinomial{n}{1}\int_0^t\frac{s^n-1}{(1+s^2)^{n/2}}ds \right).\end{equation}

Moreover for all $n\in\NN$
\begin{equation}\label{phi2nga}
\bar\varphi_{2n}(t)=(1+t^2)^{n}\left(1+\frac{1}{\pi}\sum_{k=1}^{n}\frac{1}{k}\frac{2^{2k}}{\binom{2k}{k}}\frac{t^{2k-1}+t}{(t^2+1)^{k}}\right)
\end{equation}
and
\begin{equation}\label{phi2n+1ga}
\bar\varphi_{2n+1}(t)=(1+t^2)^{n}\sum_{k=0}^{n}\binom{2k}{k}\frac{t^{2k}+t}{2^{2k}(t^2+1)^{k}}.
\end{equation}
\end{theorem}

\begin{proof}
For a fixed $n\in\NN$ consider the derivative of $\bar\varphi_{n+2}(t)$. By \eqref{rec1} one has for all $t\in\RR$
\begin{equation}\label{diffphi1}\begin{split}\bar\varphi_{n+2}'(t)&:=\sum_{k=1}^{n+2}k\bbinomial{n+2}{k}t^{k-1}=\bbinomial{n+2}{1}+
(n+2)\sum_{k=2}^{n+2}\frac{k}{n+2}\bbinomial{n+2}{k}t^{k-1}\\&=\bbinomial{n+2}{1}+
(n+2)\sum_{k=2}^{n+2}\bbinomial{n}{k-2}t^{k-1}\\&=\bbinomial{n+2}{1}+
(n+2)\sum_{k=0}^{n}\bbinomial{n}{k}t^{k+1}=\bbinomial{n+2}{1}+(n+2)t\bar\varphi_n(t).
\end{split}\end{equation}
On the other hand, Proposition \ref{tartagliagen} implies
\begin{equation}\label{recrel}
\begin{split}    \bar\varphi_{n+2}(t)&=\sum_{k=0}^{n+2}\bbinomial{n+2}{k}t^k=\sum_{k=0}^{n+2}\left(\bbinomial{n}{k-2}+\bbinomial{n}{k}\right)t^k\\
    &=\sum_{k=0}^{n+2}\bbinomial{n}{k-2}t^k+\sum_{k=0}^{n+2}\bbinomial{n}{k}t^k=\sum_{k=-2}^{n}\bbinomial{n}{k}t^{k+2}+\sum_{k=0}^{n+2}\bbinomial{n}{k}t^{k}\\
     &=\bbinomial{n}{-1}t+(t^2+1)\bar\varphi_n(t)+\bbinomial{n}{n+1}t^{n+1}\\
     &=(t^2+1)\bar\varphi_n(t)+\bbinomial{n}{-1}\left(t^{n+1}+t\right).
\end{split}\end{equation}
By differentiating in both sides we then obtain from \eqref{diffphi1} 
$$\bbinomial{n+2}{1}+(n+2)t\bar\varphi_n(t)=(1+t^2)\bar\varphi'_n(t)+2t\bar\varphi_n(t) +\bbinomial{n}{1}t^n+\bbinomial{n}{-1}.$$
Therefore $\bar\varphi_n$ is the solution of the linear Cauchy problem
\begin{equation}
    \label{CP}
    \begin{cases}
    (1+t^2)\bar\varphi'-n t\bar\varphi+\bbinomial{n}{1}(t^n-1)=0\\
    \bar\varphi(0)=1
    \end{cases}
\end{equation}
i.e.,
$$\bar\varphi_{n}(t)=(1+t^2)^{n/2}\left(1-\bbinomial{n}{1}\int_0^t\frac{s^n-1}{(1+s^2)^{n/2+1}}ds \right).$$

By an inductive argument based on \eqref{recrel}, i.e., on the recursive relation
$$\bar\varphi_{n+2}(t)=(1+t^2)\bar\varphi_n(t)+\bbinomial{n}{-1}(t^{n+1}+t),$$
 one readily get
 \begin{equation}\label{phi2n}
\bar\varphi_{2n}(t)=(1+t^2)^{n}\left(1+\sum_{k=0}^{n-1}\bbinomial{2k}{-1}\frac{t^{2k+1}+t}{(t^2+1)^{k+1}}\right)
\end{equation}
and
\begin{equation}\label{phi2n+1}
\bar\varphi_{2n+1}(t)=(1+t^2)^{n}\left(1+t+\sum_{k=0}^{n-1}\bbinomial{2k+1}{-1}\frac{t^{2k+2}+t}{(t^2+1)^{k+1}}\right).
\end{equation}

 We  use \eqref{phi2n} to prove \eqref{phi2nga}. We have
 \begin{align*}
\bar\varphi_{2n}(t)&=(1+t^2)^{n}\left(1+\sum_{k=0}^{n-1}\bbinomial{2k}{-1}\frac{t^{2k+1}+t}{(t^2+1)^{k+1}}\right)\\
&=(1+t^2)^{n}\left(1+\sum_{k=0}^{n-1}\frac{1}{2k+1}\bbinomial{2k}{1}\frac{t^{2k+1}+t}{(t^2+1)^{k+1}}\right)\\
&=(1+t^2)^{n}\left(1+\frac{1}{\pi}\sum_{k=0}^{n-1}\frac{1}{2k+1}\frac{2^{2k+1}}{\binom{2k}{k}}\frac{t^{2k+1}+t}{(t^2+1)^{k+1}}\right)\\
&=(1+t^2)^{n}\left(1+\frac{1}{\pi}\sum_{k=0}^{n-1}\frac{1}{k+1}\frac{2^{2k+2}}{\frac{4k+2}{k+1}\binom{2k}{k}}\frac{t^{2k+1}+t}{(t^2+1)^{k+1}}\right)\\
&=(1+t^2)^{n}\left(1+\frac{1}{\pi}\sum_{k=0}^{n-1}\frac{1}{k+1}\frac{2^{2k+2}}{\binom{2(k+1)}{k+1}}\frac{t^{2k+1}+t}{(t^2+1)^{k+1}}\right)\\
&=(1+t^2)^{n}\left(1+\frac{1}{\pi}\sum_{k=1}^{n}\frac{1}{k}\frac{2^{2k}}{\binom{2k}{k}}\frac{t^{2k-1}+t}{(t^2+1)^{k}}\right).
\end{align*}
 The proof of \eqref{phi2n+1ga} is similar and we omit it.
\end{proof}
We now focus on the particular value $\bar\varphi_n(1)=\sum_{k=0}^n\left[n\atop k\right]$, see also Figure \ref{finitephi}. 

\begin{theorem}
For all $n\in\NN$
\begin{equation}\label{integral}\sum_{k=0}^n\bbinomial{n}{k}=2^{n/2}\left(1+\bbinomial{n}{1}\int_0^1\frac{1-s^n}{(1+s^2)^{n/2+1}}ds \right).\end{equation}
Moreover, one also has for $n\in\NN$
\begin{equation}\label{evensum}
\sum_{k=0}^{2n}\bbinomial{2n}{k}=2^{n}\left(1+\frac{2}{\pi}\sum_{k=1}^{n} \frac{1}{k\binom{2k}{k}}2^{k}\right)
\end{equation}
and
\begin{equation}\label{oddsum}
\sum_{k=0}^{2n+1}\bbinomial{2n+1}{k}=2^{n+1}\sum_{k=0}^{n}\binom{2k}{k}\frac{1}{2^{3k}}.
\end{equation}
Finally, the following asymptotic estimate holds
\begin{equation}\label{asymbar}
\sum_{k=0}^{n}\bbinomial{n}{k}\sim 2^{n/2+1}\qquad \text{as }n\to\infty.\end{equation}
\end{theorem} 
\begin{proof}
The equalities in \eqref{integral},\eqref{evensum} and \eqref{oddsum} readily follow by  applying to $t=1$ in Theorem \ref{thmphi}.

As in the proof of Theorem \ref{thm1}, to show \eqref{asymbar} we distinguish two cases depending on the parity of $n$. If $n$ is even, we recall from \cite{lasalvezza} that 
$$W(x)=\sum_{k=1}^\infty \frac{4^k}{k\binom{2k}{k}}x^k=2\sqrt{\frac{x}{1-x}}\arctan{\sqrt{\frac{x}{1-x}}}.$$
Then we have
$$W(1/2)=\sum_{k=1}^\infty \frac{2^k}{k\binom{2k}{k}}=\frac{\pi}{2}.$$
In view of \eqref{evensum}, we get
$$\lim_{n\to+\infty} \frac{\varphi_{2n}(1)}{2^n}=1+\frac{2}{\pi}W(1/2)=2.$$
Similarly, when $n$ is odd, note that
$$Z(x)=\sum_{k=0}^\infty {\binom{2k}{k}}x^k=\sqrt{\frac{1}{1-4x}}$$
and in particular we have
$$Z(1/8)=\sum_{k=0}^\infty {\binom{2k}{k}}\frac{1}{2^{3k}}=\sqrt{2}.$$
In view of \eqref{oddsum}, we get
$$\lim_{n\to+\infty} \frac{\varphi_{2n+1}(1)}{2^n}=2W(1/8)=2\sqrt{2},$$
and this concludes the proof.
\end{proof}
\begin{figure}\centering
\includegraphics[scale=0.6]{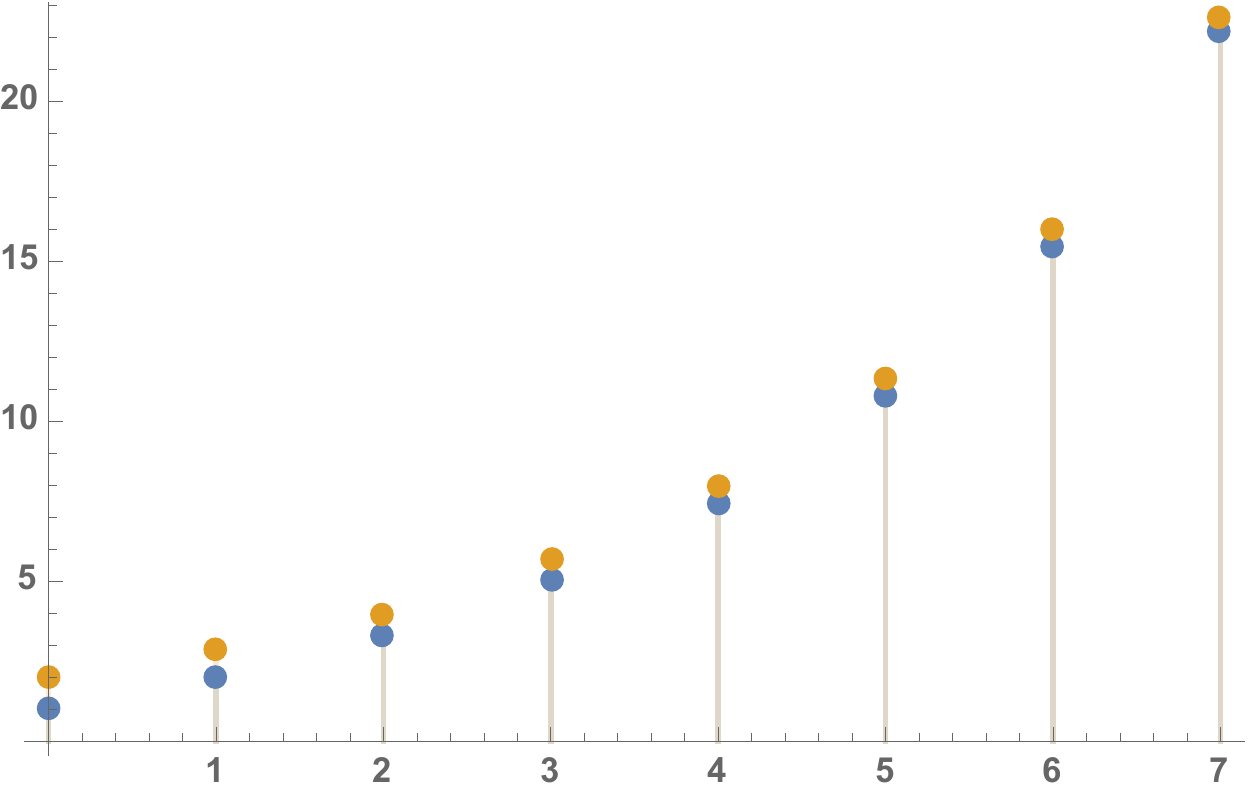}
\caption{$\bar \varphi_n(1)$ (in blue) and $2^{n/2+1}$ (in orange).\label{finitephi}}
\end{figure}
We point out that from the proof of Theorem \ref{thmphi} also follow the identies
$$\sum_{k=0}^{2n} \bbinomial{2n}{k}=\displaystyle{2^{n}\left(1+\sum_{k=0}^{n-1}\frac{1}{2k+1} \bbinomial{2k}{1}\frac{1}{2^k}\right)},$$ $$ \sum_{k=0}^{2n+1} \bbinomial{2n+1}{k}=\displaystyle{2^{n}\left(2+\sum_{k=0}^{n-1}\frac{1}{2k+2} \bbinomial{2k+1}{1}\frac{1}{2^k}\right)},$$
see Table \ref{tsum} for other, equivalent formulations. This motivates to collect in the following proposition some related finite sums.

\begin{proposition}For all $n\in\NN$
\begin{align}
&\displaystyle{\sum_{k=0}^{n-1}\frac{1}{2k+1} \bbinomial{2k}{1}}=\bbinomial{2n}{1}-\dfrac{2}{\pi}\label{s1}\\
&\displaystyle{\sum_{k=0}^{n-1}\frac{1}{2k+2} \bbinomial{2k+1}{1}}=\bbinomial{2n+1}{1}-1\label{s2}\\
&\displaystyle{\sum_{k=0}^{n} \bbinomial{k}{2}}=\dfrac{n(n+1)}{4}\label{s3}\\
&\displaystyle{\sum_{k=1}^{n} \frac{1}{k}\bbinomial{k}{2}}=\dfrac{n}{2}\label{s4}\\
&\displaystyle{\sum_{k=0}^{n} \frac{1}{k+1}\bbinomial{k}{2}}=\dfrac{(n+1)}{2}-\dfrac{1}{2}H_{n+1}\label{s5}
\end{align}
where $H_{n}:=\sum_{k=1}^n\frac{1}{k}$ is the $n$-th armonic number.
\end{proposition}
\begin{proof}
We have by \eqref{21even} and by the finite sum of inverse central binomials established in \cite{lasalvezza}
$$\displaystyle{\sum_{k=0}^{n-1}\frac{1}{2k+1} \bbinomial{2k}{1}}=\displaystyle{\frac{1}{\pi}\sum_{k=1}^n \frac{1}{k}\frac{2^{2k}}{\binom{2k}{k}}}=\displaystyle{\frac{2^{2n+1}}{\pi \binom{2n}{n}}-\frac{2}{\pi}}=\bbinomial{2n}{1}-\dfrac{2}{\pi};$$
By \eqref{21odd} (see also \cite{gouldvol3} for the finite sum of central binomials)
$$\displaystyle{\sum_{k=0}^{n-1}\frac{1}{2k+2} \bbinomial{2k+1}{1}}=\displaystyle{\sum_{k=1}^n \frac{\binom{2k}{k}}{2^{2k}}}=\displaystyle{\frac{2n+1}{2^{2n} \binom{2n}{n}}-1}=\bbinomial{2n+1}{1}-1$$
The equalities \eqref{s3}-\eqref{s5} are a direct consequence of the fact that $\left[ k\atop 2\right]=k/2$ for all $k\in\NN$.
\end{proof}

\begin{table}{\small
\begin{tabular}{|c|c|c|c|c|}\hline&&&&\\
$\displaystyle{\sum_{k=0}^n \binom{n}{k}=2^n}$&$\displaystyle{\sum_{k=0}^\infty \binom{n/2}{k}=2^{n/2}}$&$\displaystyle{\int_{-\infty}^{+\infty} \binom{n/2}{x/2}dx=2^{n/2+1}}$&$\displaystyle{\sum_{k=0}^{+\infty} \bbinomial{n}{k}\sim 2^{n/2+1}}$&$\displaystyle{\sum_{k=0}^{n} \bbinomial{n}{k}\sim 2^{n/2+1}}$\\&&&&\\\hline
\end{tabular}}
\caption{A comparison between Theorem 3.2 and earlier related versions of Binomial Theorem (note that the equality involving the integral is derived from \eqref{integralsum} by a change of variable) }
\end{table}
\begin{table}
{\small \begin{tabular}{|c||c|c|c|}\hline
Sum& \multicolumn{3}{l|}{Equivalent formulations}\\\hline&&&\\
$\displaystyle{\sum_{k=0}^{2n}\bbinomial{2n}{k}}$&$\displaystyle{2^{n}\left(1+\frac{2}{\pi}\sum_{k=1}^{n} \frac{1}{k\binom{2k}{k}}2^{k}\right)}$&$\displaystyle{2^{n}\left(1+\sum_{k=0}^{n-1} \bbinomial{2k}{-1}\frac{1}{2^k}\right)}$&$\displaystyle{2^{n}\left(1+\sum_{k=0}^{n-1}\frac{1}{2k+1} \frac{1}{2^k}\right)}$\\&&&\\\hline&&&\\&&&\\
$\displaystyle{\sum_{k=0}^{2n+1}\bbinomial{2n+1}{k}}$&$\displaystyle{
2^{n+1}\sum_{k=0}^{n}\binom{2k}{k}\frac{1}{2^{3k}}}$&$\displaystyle{2^{n}\left(2+\sum_{k=0}^{n-1} \bbinomial{2k+1}{-1}\frac{1}{2^k}\right)}$&$\displaystyle{2^{n+1}\left(2+\sum_{k=0}^{n-1}\frac{1}{2k+2} \bbinomial{2k+1}{1}\frac{1}{2^k}\right)}$
\\&&&\\\hline&&&\\&&&\\
$\displaystyle{\sum_{k=0}^{n-1}\frac{1}{2k+1} \bbinomial{2k}{1}}$&$\displaystyle{\frac{1}{\pi}\sum_{k=1}^n \frac{1}{k}\frac{2^{2k}}{\binom{2k}{k}}}$&$\displaystyle{\frac{2^{2n+1}}{\pi \binom{2n}{n}}-\frac{2}{\pi}}$&$\bbinomial{2n}{1}-\dfrac{2}{\pi}$\\&&&\\
$\displaystyle{\sum_{k=0}^{n-1}\frac{1}{2k+2} \bbinomial{2k+1}{1}}$&$\displaystyle{\sum_{k=1}^n \frac{\binom{2k}{k}}{2^{2k}}}$&$\displaystyle{\frac{2n+1}{2^{2n} \binom{2n}{n}}-1}$&$\bbinomial{2n+1}{1}-1$\\&&&\\\hline&&&\\
$\displaystyle{\sum_{k=0}^{n} \bbinomial{k}{2}}$&$\displaystyle{\sum_{k=0}^n \frac{k}{2}}$&$\dfrac{n(n+1)}{4}$&$\bbinomial{n}{2}\cdot\bbinomial{n+1}{2}$\\&&&\\
$\displaystyle{\sum_{k=1}^{n} \frac{1}{k}\bbinomial{k}{2}}$&$\displaystyle{\sum_{k=1}^n \frac{1}{2}}$&$\dfrac{n}{2}$&$\bbinomial{n}{2}$\\&&&\\
$\displaystyle{\sum_{k=0}^{n} \frac{1}{k+1}\bbinomial{k}{2}}$&$\displaystyle{\sum_{k=0}^n \frac{k}{2(k+1)}}$&$\dfrac{(n+1)}{2}-\dfrac{1}{2}H_{n+1}$&\\&&&\\\hline
\end{tabular}}
\caption{Identities involving $\varphi_{n}(1)$  -- $H_n$ is the $n$-th armonic number\label{tsum}.}
\end{table}

\end{document}